\documentclass[final,1p,times]{elsarticle}
\usepackage[latin1]{inputenc}
\usepackage{amsmath}
\usepackage{amsfonts}
\usepackage{amssymb}
\usepackage{fancyhdr}

\newtheorem{theorem}{Theorem}[section]
\newtheorem{lemma}[theorem]{Lemma}

\newenvironment{example}[1][Example]{\begin{trivlist}
\item[\hskip \labelsep {\bfseries #1}]}{\end{trivlist}}
\newenvironment{proof}[1][\textit{Proof.}]{\begin{trivlist}
\item[\hskip \labelsep {\bfseries #1}]}{\end{trivlist}}

\newcommand{\Cl}{\operatorname{Cl}}

\biboptions{comma,square}

\journal{Journal of Geometry and Physics}

\begin{document}
\begin{frontmatter}

\title{Getzler Symbol Calculus and Deformation Quantization}
\author[cm]{Camilo Mesa\fnref{label2}}
\ead{mesa@colorado.edu}
\fntext[label2]{Department of Mathematics, University of Colorado at Boulder}
\address[cm]{Department of Mathematics, University of Colorado, 
Campus Box 395, \\ 
Boulder, Colorado 80309-0395\\USA}


\begin{abstract}
In this paper we give a construction of Fedosov quantization incorporating the odd variables and an analogous formula to Getzler's pseudodifferential calculus composition formula is obtained. A Fedosov type connection is constructed on the bundle of Weyl tensor Clifford algebras over the cotangent bundle of a Riemannian manifold. The quantum algebra associated with this connection is used to define a deformation of the exterior algebra of Riemannian manifolds.
\end{abstract}

\begin{keyword}
Deformation quantization  \sep
Atiyah-Singer index theorem
\MSC[2010] 53D55  \sep 58B34 \sep 58J20

\end{keyword}

\end{frontmatter}

\section{Introduction}
Getzler in \cite{getzler1983pseudodifferential} introduced a pseudodifferential calculus of symbols on supermanifolds which simplifies the calculation of the top order symbol of the square of the Dirac operator in the proof of the Atiyah-Singer theorem on a spin manifold. For this construction, supersymmetric ideas from quantum field theory lead to the use of Clifford variables and a suitable filtration for the space of pseudodifferential symbols. These ideas combined with  more classical work of Widom \cite{widom1978families} yielded a composition formula of symbols 
\begin{equation}
\label{getzlerproduct}
 p \circ q(x,\xi)= \left. \exp \left( -\frac{1}{4}R \left( \frac{\partial}{\partial \xi},\frac{\partial}{\partial \eta}\right) \right)p(x,\xi) \wedge q(y,\eta) \right|_{y=x,\eta=\xi} + \cdots.
\end{equation}
This remarkable formula not only provided a short proof of the Atiyah-Singer index theorem, but is of great interest from the point of view of deformation theory given its distinctive geometric features such as the appearance of the curvarture operator $R \left( \partial / \partial \xi, \partial / \partial \eta \right)$. 

On the other hand, deformation quantization emerged with the goal of understanding as much as possible of quantum mechanics in terms of deformed algebra structures. The seminal paper \cite{bayen1978deformation} considers the question of defining local associative deformations (star products)  
\begin{equation}
\label{starproduct}
p \star q = pq + (\hbar/2) \lbrace p, q \rbrace + \hbar^2 B_2(p,q)+ \cdots
\end{equation}
of the algebra of smooth complex-valued functions on a symplectic manifold $(M,\omega)$, where the first order term corresponds to the Poisson bracket associated to $\omega$. Fedosov in \cite{fedosov1994simple} provided a canonical construction of star products associated with any symplectic connection on a symplectic manifold. Moreover a classification of these deformations was established via a one-to-one correspondence with formal power series $\Omega=\Omega_0 + \hbar \Omega_1 +\cdots$ of cohomology classes $\Omega_i \in H^2(M)$, where $\Omega_0$ is the class of the symplectic form. The product given by composition of symbols in the usual pseudodifferential calculus on the cotangent bundle of a Riemannian manifold is an important example of a star product. From this observation, Fedosov \cite{fedosovindex} and Nest and Tsygan \cite{NestTsyganAlgIndex95} proposed analogues of the Atiyah-Singer index theorem for the deformation quantization algebra of a given symplectic manifold.

Given the common motivation for formulas of the type (\ref{getzlerproduct}) and (\ref{starproduct}), the goal of this paper is to construct a general local theory of deformations by Getzler. Having as a proper framework a polarized symplectic manifold, we pursue this goal by means of a Fedosov type description in the particular case of the cotangent bundle of a given Riemannian manifold.  The main idea is to incorporate the Clifford algebra in Fedosov's construction to obtain a a deformation of  the algebra of formal power series in $\hbar$ with differential forms on a Riemannian manifold $X$ as coefficients.

The content and organization of the following sections is as follows. In section 2 we present the basic algebraic machinery of the construction by defining the Weyl and Clifford algebra bundles. We prove the existence of an Abelian connection on the Weyl tensor Clifford algebra bundle in section 3. Sections of the Weyl tensor Clifford algebra bundle which are flat with respect to this connection form a subalgebra. In section 4, Theorem \ref{secondtheorem} stablishes a one-to-one correspondence between this subalgebra and sections of $\Lambda(X)[[\hbar]]$. Finally, using this correspondence we define a deformation (\ref{deformation}) which is analogous to Getzler composition formula.

Finally I would like to express my sincerest gratitude to my advisor, Professor Alexander Gorokhovsky, who proposed the topic and supported me throughout this research with patience, and Professor Arlan Ramsay for many useful discussions and corrections.
\section{The Weyl and Clifford algebra bundles}

We outline the constructions of the Weyl and Clifford algebras. For proofs we refer the reader to \cite{fedosovindex} and \cite{chevalley97algebraic}. The following summary develops in the context of a general symplectic structure (resp. Riemannian structure) for the Weyl algebra (resp. Clifford algebra). For our construction however, the base space for the Weyl algebra bundle will be the cotangent bundle of a Riemannian manifold with standard symplectic structure.

Let $\hbar$ be a formal parameter. The Weyl algebra is the $\mathbb{C}[[\hbar]]$-algebra
\begin{equation}
W= \mathbb{C}[[\hat x^1, \ldots, \hat x^n, \hat \xi^1, \ldots, \hat \xi^n, \hbar]]
\end{equation}
of formal power series in $\hat x^i, \hat \xi^i$, and $\hbar$ with associative product given by the Moyal rule
\begin{equation}
f \circ g=  \left.  \exp \left( \frac{i\hbar}{2} \sum_{i=1}^n \left( \frac{\partial}{\partial \hat x^i} \frac{\partial}{\partial \hat \eta ^i} -  \frac{\partial}{\partial \hat \xi^i} \frac{\partial}{\partial \hat y^i}\right)\right) f(\hat{x},\hat{\xi},\hbar) \, g(\hat{y},\hat{\eta},\hbar) \right|_{\hat y = \hat x, \, \hat \eta = \hat \xi }.
\end{equation}
Note that $W$ is generated by $\hat{x}^i, \hat{\xi}^i$ and $\hbar$ with relations $\hat{x}^i \circ \hat{\xi}^j - \hat{\xi}^j \circ \hat{x}^i = i\hbar \delta_{ij}$. Setting the degrees of the generators as: $\vert \hat{x}^i \vert, \vert \hat{\xi}^i \vert=1$, $\vert \hbar \vert=2$, we obtain $\mathbb{Z}$-filtration of $W$. 

The symplectic group $\mathrm{Sp}_{2n}$ of linear transformations preserving the standard symplectic form $\omega=  \sum dx^i \wedge d\xi^i$ on $\mathbb{R}^{2n}$ acts on $W$ by automorphisms. To describe the corresponding action of the Lie algebra $ \mathfrak{sp}_{2n}$ on $W$ it is convenient to use a reordering $y^i$, $i=1,\ldots,2n$, of the usual coordinates $(x^i,\xi^i)\in \mathbb{R}^{2n}$ obtained by setting $y^{2i-1}=x^i, y^{2i}=\xi^i$ for $i=1,\ldots,n$. In these coordinates the symplectic form $\omega$ is written as
\begin{equation}
\omega = \frac{1}{2} \sum \omega_{ij} dy^i \wedge dy^j.
\end{equation}
Recall that $\mathfrak{sp}_{2n}$ consists of matrices $S=s_j^i$ such that $\sigma_{ij} = \omega_{ik} s^k_j$ is symmetric. Each $S \in \mathfrak{sp}_{2n}$, induces an inner derivation defined as $S(f)=[\tilde{S},f]$, for $f\in W$, where 
\begin{equation} \label{eq: derW}
\tilde{S}=\frac{1}{2\hbar} \sum_{i,j=1}^{2n} \sigma_{ij} \hat{y}^i \hat{y}^j.
\end{equation}
Thus we think of $\mathfrak{sp}_{2n} \hookrightarrow W$ as the Lie algebra of homogeneous quadratic polynomials in $W$. Moreover the kernel of the map $f \mapsto (1/\hbar)[f, \cdot]$ (Moyal product bracket) from $W$ onto the Lie algebra of derivations of $W$ is $\mathbb{C} [[\hbar,\hbar^{-1} ]]$. Therefore
\begin{equation}
\mathfrak{g}=\frac{1}{\hbar}W/ \mathbb{C} [[\hbar,\hbar^{-1} ]]
\end{equation}
is isomorphic to the Lie algebra of derivations of $W$.

Now we define the bundle of Weyl algebras on a symplectic manifold $(M, \omega)$ of dimension $2n$. Symplectic frames form a principal bundle $\mathrm{Sp}(M) \rightarrow M$ with structure group $\text{Sp}_{2n}$.  The Weyl algebra bundle over $M$ is an associated bundle to the symplectic frame bundle,
\begin{equation}
\label{WeylBundle}
\mathcal{W}= \mathrm{Sp}(M) \times_{\mathrm{Sp}_{2n}} W.
\end{equation}
If we trivialize $F_{\text{Sp}}(M)$ over $M$, $F_{\text{Sp}}(M) \rightarrow M \times \text{Sp}_{2n}$,
we see that locally, $\mathcal{W}$ is identified with $M \times W$. Locally we write sections of $\mathcal{W}$ as
\begin{equation}
\sum_{|\alpha|, k \geq 0} \hbar^k a_{k,\alpha}\hat{y}^\alpha
\end{equation}
where $\alpha \in \mathbb{N}^{2n}_{0}$ is a multi-index, $a_{k,\alpha}$ are smooth functions on $M$, $\hat{y}^\alpha=(\hat{y}^1)^{\alpha_1} \cdots(\hat{y}^{2n})^{\alpha_{2n}}$, and $\hat{y}^{2i-1}=\hat{x}^i$, $\hat{y}^{2i}=\hat{\xi}^i$, $1\leq i \leq n$. The filtration of $W$ induces a filtration on $\mathcal{W}$
\begin{equation}
\mathcal{W} \supset \cdots \supset \mathcal{W}^{-1} \supset \mathcal{W}^0 \supset \mathcal{W}^1 \supset \mathcal{W}^2 \supset \cdots,
\end{equation}
where $\mathcal{W}^j = \lbrace \sum_{2k+ |\alpha| \geq j} \hbar^k a_{k,\alpha}\hat{y}^\alpha  \rbrace$.

Now we define the Clifford algebra bundle over a Riemannian manifold $X$. The facts in the following summary can be found in \cite{atiyah1973heat} and \cite{chevalley97algebraic} with the exception that the deformation parameter $\hbar$ is introduced as part of the definition of the Clifford algebra. If $V$ is a real or complex vector space of dimension $n$ with inner product $\langle \cdot, \cdot \rangle$, then the Clifford algebra, denoted by $\mathrm{Cl}_{\hbar}(V)$, is the algebra generated by $V$ with relations 
\begin{equation}
\label{cliffordrelation}
vu+uv=-2\hbar \langle v, u\rangle.
\end{equation}
One can also think of $\mathrm{Cl}_{\hbar}(V)$ as the quotient of the tensor algebra $\mathcal{T}(V)=\bigoplus_{j \geq 0} V \otimes \cdots \otimes V$ by the two-sided ideal $\mathcal{I}$ generated by all elements of the form $v \otimes v + \hbar\Vert v \Vert ^2$ for $v \in V$. There is an increasing filtration $\mathrm{Cl}_{\hbar}(V)= \cup_{k} \mathrm{Cl}_{\hbar}^{k}(V)$ induced by the order gradation from the tensor algebra. An element $v \in \mathrm{Cl}_{\hbar}(V)$ is of order $k$ if there is $u \in \bigoplus_{j \geq 1}^k V \otimes \cdots \otimes V$ such that $v=[u]$. We define the order of $\hbar \in \Cl_{\hbar}(V)$ to be 2. If we reduce modulo 2 the $\mathbb{Z}$-grading of $\mathcal{T}(V)$ we see that $\mathcal{T}(V)$ is a superalgebra. Since the ideal $\mathcal{I}$ is generated by even elements, the Clifford algebra is itself a superalgebra $\mathrm{Cl}_{\hbar}(V)=\mathrm{Cl}_{\hbar}^{+}(V) \oplus \mathrm{Cl}_{\hbar}^{-}(V)$.

The space $\frac{1}{\hbar}\Cl_{\hbar}^2(V)$ forms a Lie algebra with Clifford product bracket and its adjoint action on $V$ yields the following result. 

\begin{lemma} The adjoint action $\tau$ of the Lie algebra
\[
\mathfrak{spin}_{\hbar}(V):=\frac{1}{\hbar}\Cl_{\hbar}^2(V)
\]
on $V$ induces a Lie algebra isomorphism between $\mathfrak{spin}_{\hbar}(V)$ and $\mathfrak{so}(V)$.
\end{lemma}
\begin{proof}
Let $e_1,\ldots,e_n$ be an orthonormal basis for the vector space $V$. We start by showing that $[a,v]\in V$ for all $a \in \mathfrak{spin}_{\hbar}(V)$ and $v \in \Cl_{\hbar}^1(V)=V$. It is enough to check the following two cases. First assume $a=\frac{1}{\hbar} e_i e_j$ and $v=e_k$ with $i,j,k$ all different. Then
\begin{eqnarray*}
[a,v] & = & \frac{1}{\hbar} ( e _i e_j e_k - e_k e_i e_j )\\
& = &  \frac{1}{\hbar} ( e _i e_j e_k - e_i e_j e_k )\\
& = & 0.
\end{eqnarray*}
The second case is $a=\frac{1}{\hbar} e_i e_j$ and $v=e_i$. In this case we have
\begin{eqnarray*}
[a,v] & = & \frac{1}{\hbar} (e _i e_j e_i - e_i e_i e_j) \\
& = & \frac{1}{\hbar}( -e _i e_i e_j - e_i e_i e_j ) \\
& = & 2e_j.
\end{eqnarray*}
Also note that $\tau [a,b]= [\tau (a), \tau (b)]$ so it follows that $\tau$ defines a Lie algebra homomorphism from $\mathfrak{spin}_{\hbar}(V)$ into $\mathfrak{gl}(V)$. Moreover, $\tau$ actually maps into $\mathfrak{so}(V)$ since by the Jacobi identity
\begin{eqnarray*}
\langle \tau(a)v, u \rangle + \langle v, \tau(a) u \rangle & = & -\frac{1}{2\hbar}[[a,v],u]- \frac{1}{2\hbar}[v,[a,u]] \\
& = & -\frac{1}{2\hbar} [a,[v,u]] \\
& = & 0.
\end{eqnarray*}
Since $\tau$ is injective, it is an isomorphism since $\mathfrak{so}(V)$ and $\mathfrak{spin}_{\hbar}(V)$ are both of dimension $n(n-1)/2$.
\qed
\end{proof}
Working with an oriented orthonormal basis $e_1, \ldots, e_{n}$ of  $V$  one shows that given $A \in \mathfrak{so}(V)$ its corresponding element under $\tau$ is
\begin{equation} 
\label{eq:tauinverseh}
\tau^{-1}(A)=\frac{1}{2\hbar}\sum_{i<j} \langle Ae_i,e_j \rangle e_i e_j \in \mathfrak{spin}_{\hbar}(V).
\end{equation}
The exterior algebra $\Lambda V$ is a module over the Clifford algebra $\Cl_{\hbar}(V)$. Given $v\in V$ its action $c_{\hbar}(v)$ on $\Lambda V$ is 
\begin{equation}
c_{\hbar}(v)\alpha = \hbar(\epsilon(v)\alpha - \iota(v) \alpha),
\end{equation} 
where $\epsilon$ and $\iota$ denote the exterior and interior derivations in $\Lambda V$. This action respects the relation in the definition of the Clifford algebra since for $v,u \in V$ and $\alpha \in \Lambda V$ we have
\begin{eqnarray}
(c_{\hbar}(vu)+c_{\hbar}(uv))\alpha & = & \hbar \left( \epsilon(v)\epsilon(u) - \iota(v)\epsilon(u) -\epsilon(v)\iota(u) + \iota(v)\iota(u)  \right. \\
& + & \left. \epsilon(u)\epsilon(v) - \iota(u)\epsilon(v) -\epsilon(u)\iota(v) + \iota(u)\iota(v) \right)\alpha \\
& = & -\hbar(\iota(v)\epsilon(u) + \epsilon(v)\iota(u) + \iota(u)\epsilon(v) +\epsilon(u)\iota(v) )\alpha \\
& = & -2\hbar  \langle v, u \rangle \alpha.
\end{eqnarray}
Here we used the relation $\epsilon(v)\iota(u) + \iota(u) \epsilon(v) = \langle v, u \rangle$ on $\Lambda V$. The symbol map
\begin{equation}
\label{cliffordsymbol}
\sigma_{\hbar}: \Cl_{\hbar}(V) \rightarrow \Lambda V
\end{equation}
given by $\sigma_{\hbar}(v) = c_{\hbar}(v) \, 1$ is an isomorphism of vector spaces, where 1 denotes the identity in $\Lambda V$. Note that if $e_1,\ldots,e_n$ is an orthonormal basis of $V$ then 
$\sigma_{\hbar} (e_{i_1} e_{i_2} \cdots e_{i_k})= \hbar^k e_{i_1}\wedge e_{i_2} \wedge \cdots \wedge e_{i_k}$.
On the other hand, the quantization map 
\begin{equation}
\theta_{\hbar}: \Lambda V \rightarrow \Cl_{\hbar}(V)
\end{equation}
is rescaled as follows $\theta_{\hbar} (e_{i_1}\wedge e_{i_2} \wedge \cdots \wedge e_{i_k}) = \hbar^{-k}e_{i_1} e_{i_2} \cdots e_{i_k}$.

If $X$ is a Riemannian manifold of dimension $n$, the Clifford bundle $\Cl_{\hbar}$ over $X$ is an associated bundle to the orthonormal frame bundle $\mathrm{O}(X)$,
\begin{equation}
\label{CliffordBundle}
\Cl_{\hbar}= \mathrm{O}(X) \times_{\mathrm{O}(n)} \Cl_{\hbar}(\mathbb{R}^n).
\end{equation}
Hence $\Cl_{\hbar}$ is the bundle of algebras over $X$ whose fiber at $x\in X$ is the Clifford algebra $\Cl_{\hbar}(T^*_xX)$ of the inner product space $T_x^*X$.

\section{An Abelian connection on the Weyl algebra with fermionic variables}

In order to relate the geometry of the given Riemannian manifold $X$, fermionic variables corresponding to the exterior algebra $\Lambda T^*_xX$ at each $x$ are combined with the Weyl algebra. We prove that the Weyl tensor Clifford algebra bundle over the cotangent bundle of $X$ has a canonical connection which can be flattened using Fedosov recursive method. 

Let $X$ be a Riemannian manifold of dimension $n$. Consider the cotangent bundle $\pi:M = T^{*}X \rightarrow X$ with canonical symplectic form $\omega$. Let $\mathcal{W}$ be the Weyl algebra bundle on $M$ and use the trivial $\mathbb{Z}_2$-grading, $\mathcal{W}^{+}=\mathcal{W}$ and $\mathcal{W}^{-}=0$. Pull back the bundles $\Lambda T^*X$ and $\Cl_{\hbar}$ to $M$ via $\pi$ and denote them as $ \Lambda= \pi^* (\Lambda T^{*}X)$, $\mathcal{C}=\pi^* (\Cl_{\hbar})$. Now consider the bundle 
\begin{equation}
\mathcal{W} \otimes_{C^\infty(M)} \mathcal{C} \rightarrow M.
\end{equation}
Note that $\mathcal{W \otimes C}$ forms an associative algebra bundle with respect to the Moyal-Clifford product and has a $\mathbb{Z}$-filtration inherited from $\mathcal{W}$ and $\mathcal{C}$:
\begin{equation}
\left( \mathcal{W \otimes C} \right)_{p} = \bigoplus_{i+j=p}\mathcal{W}^i \otimes \mathcal{C}^j.
\end{equation}
A natural parity is induced on this bundle, $ (\mathcal{W}\otimes \mathcal{C})^+ = \mathcal{W} \otimes \mathcal{C}^+$ and  $(\mathcal{W}\otimes \mathcal{C})^{-} = \mathcal{W} \otimes \mathcal{C}^{-}$. Differential forms $\mathcal{A}^\bullet (M,\mathcal{W} \otimes \mathcal{C})$ with values in $\mathcal{W} \otimes \mathcal{C}$ also form a superalgebra with the usual parity convention. Henceforth, the bracket $[\cdot,\cdot]$ on this algebra is to be understood as a supercommutator.

Sections of this bundle are locally of the form $ a(\hbar,\hat{y},e)=\sum_{|\alpha|,|\beta|,k \geq 0} \hbar^k a_{k,\alpha,\beta}\hat{y}^\alpha e_\beta$, where $\alpha \in \mathbb{N}^{2n}_{0}$ and $\beta \in \mathbb{N}^{n}_{0}$ are multi-indices. Note that the Clifford symbol map lifts to a vector bundle isomorphism $\sigma_{\hbar}: \Gamma(M,\mathcal{C}) \rightarrow \Gamma(M,\Lambda) $. Also, since $\mathcal{W}$ is unitary, we have an embedding $ \theta_\hbar:\Gamma(M,\Lambda) \hookrightarrow \Gamma(M,\mathcal{W}\otimes \mathcal{C})$.

From (\ref{CliffordBundle}), we see that $\Cl_{\hbar}$ inherits a connection, the Levi-Civita connection $\nabla^{\Cl}$ which respects the Clifford product, $\nabla^{\Cl}(ab)=(\nabla^{\Cl} a)b+a (\nabla^{\Cl} b)$ for $a,b\in \Gamma(X,\Cl_{\hbar})$. We pull back this connection to $\mathcal{C}$ via $\pi$, and use (\ref{eq:tauinverseh}) to write it locally in an orthonormal frame as
\begin{equation}
\nabla^{\mathcal{C}}=d+\frac{1}{4\hbar}[\Gamma_{ij}^{k} e_j e_k dx^i, \cdot ]
\end{equation}
where  $\Gamma_{ij}^{k}$ are the Christoffel symbols of the Levi-Civita connection and $e_i$ correspond to a local orthonormal coframe on $X$.

On the tangent bundle of the symplectic manifold $M=T^*X$, consider a symplectic connection $\nabla$. This connection is not unique or canonical but it has some desired properties such as its compatibility with the Moyal product. Note that from (\ref{WeylBundle}), this connection will induce a connection on $\mathcal{W}$ such that $\nabla(a  \circ b)=(\nabla a)\circ b+a \circ (\nabla b)$ for $a,b\in \Gamma(M,\mathcal{W})$. For our construction however, there is a specific symplectic connection on the tangent bundle of $M$ obtained from the original Riemannian manifold $X$ data. There is a one-to-one correspondence between connections on $TM$ that are torsion free and connections preserving the symplectic form, see \cite{gelfand1998fedosov}. The Levi-Civita connection on $X$ induces a torsion free connection on $TM$, and its associated symplectic connection on $TM$ via this correspondence will be denoted by $\nabla$. In the corresponding Darboux coordinates $(x^i,\xi^i)$ on $M$, one uses (\ref{eq: derW}) to locally write the connection on $\mathcal{W}$ induced by $\nabla$ as
\begin{equation}
\nabla = d + \frac{i}{2\hbar}[\tilde{\Gamma}_{jki}\hat{y}^j\hat{y}^k d(x,\xi)^i, \cdot]
\end{equation}
where $\tilde{\Gamma}_{jki} = \omega_{jl} \tilde{\Gamma}^l_{ki}$ are the Christoffel symbols of the symplectic connection $\nabla$ and $d(x,\xi)^i=dx^1+d\xi^1+\cdots + dx^n +d\xi^n$.

The Koszul complex of $\mathcal{W}$ is acyclic, see \cite{fedosov1994simple} and its differential $\delta$ can be extended as  $\delta: \mathcal{A}^{\bullet}(M,(\mathcal{W} \otimes \mathcal{C})_{\bullet}) \rightarrow \mathcal{A}^{\bullet +1}(M,(\mathcal{W} \otimes \mathcal{C})_{\bullet -1})$ since $\mathcal{C}$ is unitary. This operator will be written locally as
\begin{equation}
\label{delta}
\delta = \frac{1}{\hbar}[\omega_{ij}\hat{y}^i d(x,\xi)^j, \cdot] = d(x,\xi)^i \frac{\partial}{\partial \hat{y}^i}.
\end{equation}

Consider also the operator $\delta^{-1}: \mathcal{A}^{\bullet}(M,(\mathcal{W} \otimes \mathcal{C})_{\bullet}) \rightarrow \mathcal{A}^{\bullet -1}(M,(\mathcal{W} \otimes \mathcal{C})_{\bullet +1})$ defined as
\begin{equation}
\delta^{-1} = \hat{y}^i \iota  \left( \frac{\partial}{\partial (x,\xi)^i} \right)
\end{equation}
where  $\frac{\partial}{\partial (x,\xi)^i} = \frac{\partial}{\partial x^1} +\frac{\partial}{\partial \xi^1}+ \cdots + \frac{\partial}{\partial x^n} +\frac{\partial}{\partial \xi^n}$.  A direct check shows that the homotopy relation 
\[
\delta \delta^{-1} + \delta^{-1} \delta = (p+r) \mathrm{Id}
\]
holds on monomials of the form $\hat{y}^{\alpha}e_{\beta}dx^{\gamma}\in \mathcal{A}^r(M,\mathcal{W}_p \otimes \mathcal{C})$ where $\alpha, \beta, \gamma$ are multi-indices. It follows that the complex  $(\mathcal{A}^{\bullet}(M,(\mathcal{W} \otimes \mathcal{C})_{\bullet}),\delta)$ is also acyclic.

Now we show how to construct a $(\mathfrak{g}\otimes \mathcal{C})$-valued connection $D$ on the bundle $\mathcal{W \otimes C} \rightarrow M$ which has central curvature. Such connection 
\begin{equation*}
D: \Gamma(M, \mathcal{W \otimes C}) \rightarrow \mathcal{A}^{1} (M,\mathcal{W \otimes C})
\end{equation*}
is of the form
\begin{equation*}
D= \nabla \otimes 1 + 1 \otimes \nabla^{\mathcal{C}} + [r,\cdot]
\end{equation*}
where $r=r_{-1}+r_0+r_1+r_2+\cdots \in \mathcal{A}^1(M,\mathfrak{g}\otimes\mathcal{C})$, $r_{-1}=\frac{1}{\hbar}\omega_{ij}\hat{y}^i d(x,\xi)^j $ induces the derivation described in (\ref{delta}), $r_0=0$, and $r_k$ for $k\geq 1$ are constructed so that $D^2 \in \mathcal{A}(M)[[\hbar]]$ is central. We show how to find $r$ so that $D^2a=0$ for any section $a \in \Gamma(M, (\mathcal{W}\otimes \mathcal{C})_p)$ and any integer $p$. Read the equation
\begin{equation}\label{eq:D2}
D^2  = \nabla^{2}  - [\delta,r] + [\nabla, r] + {\nabla^{\mathcal{C}}}^2  + \left[\nabla^{\mathcal{C}},r \right] + r^2=0
\end{equation}
in each degree and use the acyclicity of the complex $(\mathcal{A}^{\bullet}(M,(\mathcal{W} \otimes \mathcal{C})_{\bullet}),\delta)$ to solve recursively for the operators $r_k$ as follows.
\begin{itemize}
\item \textit{Degree} $p-2$. $\delta^2=0$.
\item \textit{Degree} $p-1$. $[\delta,\nabla]+[\delta,\nabla^{\mathcal{C}}]=0$.  The first term is known to be equal to zero from \cite{fedosov1994simple}. The second term is equal to zero as well since Weyl and Clifford variables commute.
\item \textit{Degree} $p$. $-[\delta,r_1] + \nabla^2 + {\nabla^\mathcal{C}}^2 = 0$. Since $[\delta, \nabla^2]=[\delta, \nabla]\nabla + \nabla[\delta, \nabla]$, 
and each of these terms are equal to zero, we have $[\delta,\nabla^2]=0$ . Also $[\delta, {\nabla^{\mathcal{C}}}^2]=0$ since ${\nabla^{\mathcal{C}}}^2$ contains no Weyl variables. Therefore $[\delta,\nabla^2 + {\nabla^\mathcal{C}}^2] = 0$ as needed. A local expression for $r_1$ can be computed using the homotopy $\delta^{-1}$,
\begin{align*}
r_1 & =  \delta^{-1} (\nabla^2 + {\nabla^\mathcal{C}}^2 ) \\
& = \frac{i}{16\hbar} \tilde{R}_{jkil} \left( \hat{y}^i \hat{y}^j \hat{y}^k d(x,\xi)^l - \hat{y}^l \hat{y}^j \hat{y}^k d(x,\xi)^i \right) \\ 
& + \frac{1}{60 \hbar} \left( R^{k} _{jil} \hat{y}^{2i-1} e_j e_k dx^l - R^{k} _{jil} \hat{y}^{2l-1} e_j e_k dx^i \right) \\
& =   \frac{i}{8\hbar} \tilde{R}_{jkil}  \hat{y}^j \hat{y}^k \hat{y}^i d(x,\xi)^l + \frac{1}{30 \hbar}  R^{k} _{jil} \hat{y}^{2i-1} e_j e_k dx^l.
\end{align*}
Here we used symmetry relations of the curvature tensor of the symplectic and the Levi-Civita connections.

\item \textit{Degree} $p+1$. $-[\delta,r_2] + [\nabla,r_1] + [\nabla^{\mathcal{C}},r_1]=0$. Using the Jacobi identity and substituing recursively, we see that $[ \delta, [\nabla +\nabla^{\mathcal{C}}, r_1]]  =  [[\delta, \nabla +\nabla^{\mathcal{C}}], r_1] + (-1)^{|\delta||\nabla + \nabla^{\mathcal{C}}|}[\nabla + \nabla^{\mathcal{C}}, [\delta,r_1]]=0$. Hence $r_2$ exists and a local expression for it is obtained from $r_2= \delta^{-1}([\nabla + \nabla^{\mathcal{C}},r_1])$.
\end{itemize}

In general, the procedure can be carried out in each degree if the non-homogeneous local section
\[
\rho = \nabla^2 +{\nabla^{\mathcal{C}}}^2 + [\nabla,r] + [\nabla^{\mathcal{C}},r]+r^2
\]
is in the kernel of $\delta$. Using the Bianchi identity and substituing recursively this can be verified as follows
\begin{eqnarray*}
[\delta,\rho] & = &[\delta, \nabla^2 +{\nabla^{\mathcal{C}}}^2] +[\delta, [\nabla +\nabla^{\mathcal{C}},r]]+[\delta, r^2] \\
& = & 0 + [[\delta,\nabla +\nabla^{\mathcal{C}}],r]+ (-1)^{|\nabla +\nabla^{\mathcal{C}}||\delta|}[\nabla +\nabla^{\mathcal{C}},[\delta,r]]+[[\delta ,r], r] \\
& = & -[\nabla +\nabla^{\mathcal{C}},[\delta,r]]-[r, [\delta, r]]\\ 
& = & -[\nabla +\nabla^{\mathcal{C}},\rho] - [r, \rho]\\
& = & -[\nabla^2 +{\nabla^{\mathcal{C}}}^2 ,r]- [\nabla + \nabla^{\mathcal{C}},r^2]-[r,\nabla^2 +{\nabla^{\mathcal{C}}}^2] - [r, [\nabla +\nabla^{\mathcal{C}},r]] - [r, r^2] \\
& = & - [\nabla^2 +{\nabla^{\mathcal{C}}}^2,r]-[[\nabla +\nabla^{\mathcal{C}},r],r]+[\nabla^2 +{\nabla^{\mathcal{C}}}^2,r]+[[\nabla +\nabla^{\mathcal{C}},r],r] \\
& = & 0.
\end{eqnarray*} 

\begin{theorem}
\label{firsttheorem}
Given
\[
\Omega_{\hbar} = \omega_0 + \hbar\omega_1 + \hbar^2\omega_2 + \cdots \in \mathcal{A}^2(M,\mathbb{C})[[\hbar]]
\]
such that $d\omega_i=0$, where $\omega_0=\omega$ is the canonical symplectic form on $M$, there exists a $1$-form $r= r_{-1} + r_0 + r_1+ \cdots$ with values in $\mathcal{W} \otimes \mathcal{C}$ where
\[
r_k: \Gamma(M, (\mathcal{W \otimes C})_p) \rightarrow \mathcal{A}^{1}(M, (\mathcal{W \otimes C})_{p+k}) 
\]
such that the connection
\[
D= \nabla \otimes 1 + 1 \otimes \nabla^{\mathcal{C}} + [r,\cdot]
\]
has curvature $D^2=\Omega_{\hbar}$.

\end{theorem}

\section{Flat sections and deformation}

There is a one-to-one correspondence between formal power series in $\hbar$ with coefficients in $\Lambda^{\bullet}$ and flat sections of the connection $D$ from Theorem \ref{firsttheorem}. Here the subalgebra of flat sections, sometimes refered to as the quantum algebra,  plays the role of the pseudodifferential operator algebra, while the exterior algebra plays the role of the symbol space.

The projection of  $\mathcal{W}$ onto its center $\tilde{\sigma}:\Gamma(M,\mathcal{W})\rightarrow C^{\infty}(M) [[\hbar]]$ can be extended to $\Gamma(M,\mathcal{W}\otimes \mathcal{C})$ simply by putting $\tilde{\sigma}=\tilde{\sigma}\otimes \mathrm{Id}$. Notice that this extension now takes values in $\Gamma(M,\mathcal{C})$. The composition of this map with the Clifford symbol map $\sigma_{\hbar}$ gives a symbol map
\[
\Sigma=\sigma_\hbar \tilde{\sigma}:\mathcal{W}\otimes \mathcal{C} \rightarrow \Lambda [[\hbar]] 
\]
The quantization map 
\[
\Theta:\Lambda [[\hbar]]  \rightarrow \ker(D)
\]
is constructed by extending any given section of $\Lambda[[\hbar]]$ to a flat section of $D$. Let $a = a_0 + \hbar a_1 + \hbar^2a_2 + \cdots$ 
be a section of the bundle $\Lambda [[\hbar]] $. Quantize $a$ using $\theta_\hbar$ and then extend $\theta_\hbar (a) \in \Gamma(M, (\mathcal{W \otimes C})_\bullet)$ to a flat section
\[
\Theta(a)=A= A_{\bullet} + A_{\bullet+1} + A_{\bullet+2} + \cdots
\]
of the bundle $\mathcal{W \otimes C}$ where $A_{\bullet}= \theta_\hbar (a) $ and $A_{\bullet+k} \in \Gamma(M, (\mathcal{W \otimes C})_{\bullet+k})$. 
The terms $A_{\bullet+k}$ in the expansion of $A$ are constructed using a recursive procedure as in \cite{fedosov1994simple}. 

\begin{example}
We show how to construct the first terms in the extension of a $1$-form. Let $a=a(x,\xi)dx^s$ be a section of $\Lambda^{1} [[\hbar]] $. First, note that $A_{-1}=\theta_\hbar (a) = \frac{1}{\hbar}a(x,\xi)e_s$ is not in the center of $\mathcal{W \otimes C}$. The flatness equation $DA=0$ is written in each degree of the filtration of $\mathcal{W\otimes C}$. This yields equations for each term $A_k$ which again can be solved using the acyclicity of the complex $(\mathcal{A}^\bullet(M,\mathcal{W\otimes C}),\delta)$.

\begin{itemize}
\item \textit{Degree} -1: 
$(\nabla \otimes 1 )A_{-1} + (1\otimes \nabla^{\mathcal C})A_{-1} - \delta A_{0} = 0$. As before, if $(\nabla \otimes 1 )A_{-1} + (1\otimes \nabla^{\mathcal C})A_{-1}$ is in the kernel of $\delta$, then $A_0$ exists and can be solved for using the homotopy $\delta^{-1}$. But this is clear since
since $(\nabla \otimes 1 )A_{-1}$ and $(1\otimes \nabla^{\mathcal C})A_{-1}$ do not contain Weyl variables. 
The computation for $A_{0}$ in local coordinates shows that $A_{0} \in \Gamma(M,\mathcal{W}_0 \otimes \mathcal{C}_0) = \Gamma(M, (\mathcal{W \otimes C})_0)$.
\item \textit{Degree} 0: $ (\nabla \otimes 1 )A_{0} + (1\otimes \nabla^{\mathcal C})A_{0} + [r_1,A_{-1}] - \delta A_1 = 0$ can be solved for $A_1$ if 
\begin{equation} \label{eq: degree2}
\delta ((\nabla \otimes 1 )A_{0} + (1\otimes \nabla^{\mathcal C})A_{0} + [r_1,A_{-1}])=0.
\end{equation}
We use Jacobi's identity to check (\ref{eq: degree2}).  First note that
\begin{align*}
\delta ((\nabla \otimes 1)A_{0}) & =  [[\delta,\nabla],A_{0}] + (-1)^{|\delta||\nabla|}[\nabla,[\delta, A_{0}]] \\
& =  0-[\nabla,[\delta,A_{0}]]\\
& =  -[\nabla, \nabla A_{-1} + \nabla^{\mathcal{C}} A_{-1}]\\
& =  -[\nabla,[\nabla^{\mathcal{C}}, A_{-1}]].
\end{align*}
Now we look at the second term in (\ref{eq: degree2}):
\begin{align*}
\delta((1\otimes \nabla^{\mathcal{C}})A_{0}) & =  [[\delta, \nabla^{\mathcal{C}}], A_{0}]] + (-1)^{|\delta||\nabla^{\mathcal{C}}|} [\nabla^{\mathcal{C}},[\delta, A_{0}]]\\
& =  0 - [\nabla^{\mathcal{C}},[\delta,A_{0}]] \\
& =  -[\nabla^{\mathcal{C}},\nabla A_{-1} + \nabla^{\mathcal{C}} A_{-1}] \\
& =  -[\nabla^{\mathcal{C}},[\nabla, A_{-1}]]- [\nabla^{\mathcal{C}},[\nabla^{\mathcal{C}},A_{-1}]] \\
& =  [\nabla,[\nabla^{\mathcal{C}}, A_{-1}]] - [(\nabla^{\mathcal{C}})^2, A_{-1}].
\end{align*} 
Finally the third term in (\ref{eq: degree2}) yields
\begin{align*}
\delta([r_1,A_{-1}]) & = [[\delta,r_1], A_{-1} ] + (-1)^{|\delta||r_1|}[r_1,[\delta,A_{-1}]] \\
& =  [\nabla^2 + {\nabla^{\mathcal{C}}}^2,A_{-1}] + 0\\
& =  0 + [(\nabla^{\mathcal{C}})^2,A_{-1}].
\end{align*} 

Equation (\ref{eq: degree2}) is verified after adding these terms. 

\end{itemize}

\end{example}

\begin{theorem}
\label{secondtheorem} 
Given a section
\begin{equation}
\label{flatsectheo}
a = a_0 + \hbar a_1 + \hbar^2a_2 + \cdots
\end{equation}
of the bundle $\Lambda[[\hbar]]$ there exists a unique section $\Theta(a)=A\in \ker(D)$ such that $\Sigma(A)=a$. 
\end{theorem}

\begin{proof}
It is enough to prove the statement for a local section $a=a_\alpha (x,\xi)dx^\alpha \in \Gamma(M,\Lambda^{r})$ of the type (\ref{flatsectheo}) with $a_1=a_2=\cdots=0$, where $(x,\xi)$ are local cotangent coordinates, $\alpha=(i_1,\ldots,i_r)$ is a multi-index.  We argue by induction. Suppose we can extend $A$ up the $n$-th term $A_s + A_{s+1} + \cdots + A_n$,  where $A_s=\theta_\hbar(a)=\frac{1}{\hbar^r}a_\alpha e_\alpha$  and $s=r - 2^r$. Note that the flatness condition $DA=0$ in degree $j$ is 
\begin{equation}\label{eq:generalflat}
(\nabla \otimes 1 )A_j + (1\otimes \nabla^{\mathcal C})A_j + \sum_{k=1}^{j-1}[r_k,A_{j-k}] - \delta A_{j+1} = 0.
\end{equation}
Hence, to prove that the term $A_{n+1}$ exists, we verify that 
\[
\phi_n =(\nabla \otimes 1 )A_n + (1\otimes \nabla^{\mathcal C})A_n + \sum_{k=1}^{n-1}[r_k,A_{n-k}]
\]
is in the kernel of $\delta$. To simplify the computation we introduce the following notation
\[
(D+ \delta)_k =
\begin{cases}
\nabla +\nabla^{\mathcal C}, & k=0, \\
r_1, & k=1, \\
r_2, & k=2, \\
\vdots \\
\end{cases}
\]
Denote the homogeneous part of degree $k$ of $D^2$ by $\Omega_k$. From (\ref{eq:D2}) we see that $[\delta,r_k]= \Omega _k$. Use equation (\ref {eq:generalflat}) for $0 \leq j < n+1$ and the Jacobi identity to obtain

\begin{align*}
[\delta,\phi_n] & =  -[\nabla +\nabla^{\mathcal{C}},[\delta, A_n]]+\sum_{k=1}^{n-1}[\delta,[r_k,A_{n-k}]]\\ 
& =   -[\nabla +\nabla^{\mathcal{C}}, [\delta, A_n]]+\sum_{k=1}^{n-1}[[\delta,r_k],A_{n-k}]  - [r_k,[\delta,A_{n-k}]] \\
& =  -\sum_{k=0}^{n-1}[(D+\delta)_k, (D+\delta)A_{n-k-1}]+\sum_{k=1}^{n-1}[[\delta,r_k],A_{n-k}]\\
& =  -\sum_{k=0}^{n-1}[\Omega_k, A_{n-k-1}]+\sum_{k=1}^{n-1}[[\delta,r_k],A_{n-k}] \\
& =  0.
\end{align*}
\qed
\end{proof}

Now we use the previous results to define star product of sections of $\Lambda [[\hbar]]$. By this we mean a $\mathbb{R} \left(( \hbar \right))$-bilinear associative multiplication law  
\begin{equation}
\label{deformation}
\star:\Lambda[[\hbar]]  \otimes_{\mathbb{R}\left(( \hbar \right))} \Lambda [[\hbar]]  \rightarrow \Lambda [[\hbar]]
\end{equation}
which we write as $ p \star q = p \wedge q + \hbar \beta_1(p,q)+\hbar^2 \beta_2(p,q) +\cdots$,  for $p,q\in \Gamma(M,\Lambda[[\hbar]])$, where each $\beta_i$ is bilinear.  Using the Moyal-Clifford product $\circ$ on $\mathcal{W \otimes C}$ we define
\begin{equation}
p\star q := \Sigma(\Theta(p)\circ \Theta(q)).
\end{equation}
\begin{example}
The flat section $P=P_1 + P_2 + P_3 + \cdots $, associated to $p=p(x,\xi)dx^r \in \Gamma(M,\Lambda^1)$ where $r\in \lbrace 1,\ldots,n \rbrace$ is fixed, is given by
\begin{eqnarray*}
P & = & \frac{1}{\hbar}pe_r + \frac{1}{\hbar}(\partial_ip \hat{y}^i e_r + p\Gamma_{ir}^{k} \hat{y}^{2i-1} e_k) + \frac{1}{2\hbar}(\partial_l \partial_ip \hat{y}^l \hat{y}^i e_r + \partial_l (p) \Gamma_{ir}^k \hat{y}^l\hat{y}^{2i-1} e_k + (\partial_mp )\tilde{\Gamma}_{li}^m \hat{y}^l \hat{y}^i e_r \\
& + &(\partial_ip )\Gamma_{lr}^k \hat{y}^{2l-1}\hat{y}^i e_k + p\tilde{\Gamma}_{lj}^m \Gamma_{ir}^k \hat{y}^l \hat{y}^{2i-1} e_k +\frac{1}{20\hbar}pR_{rli}^k \hat{y}^{2l-1} \hat{y}^{2i-1} e_k) + \cdots.
\end{eqnarray*} 
Suppose $q=q(x,\xi)dx^s \in \Gamma(M,\Lambda^1)$ and $Q$ is its associated flat section. Fix a point $\mathbf{x}\in X$ and let $(x^1,\ldots,x^n)$ be a normal coordinate system at $\bold{x}$. Using the fact that the Christoffel symbols of the Levi-Civita connection are zero at $\bold{x}$ we have
\begin{align*}
\left. p \star q \right|_{\bold{x}}& =  \Sigma( P \circ Q) \\
& =  p \, q\, dx^r dx^s + \frac{1}{40}\left( \frac{\partial^2 p}{\partial \xi^l \partial \xi^i} q - p \frac{\partial^2 q}{\partial \xi^l \partial \xi ^i}  \right) R_{rli}^s + \text{higher order terms}
\end{align*}
Note that these first two terms correspond to the leading symbol of $p\star q$ obtained by setting $\hbar=0$.
\end{example}

\section{References}
\bibliographystyle{model1-num-names}
\bibliography{references}
\end{document}